\documentclass[a4paper,12pt]{article}
\usepackage{amssymb,amsmath,amsthm,latexsym}
\usepackage{amsfonts}
\usepackage{amsfonts}
\usepackage{graphicx}
\usepackage[pdftex,bookmarks,colorlinks=false]{hyperref}
\usepackage{verbatim}
\usepackage[font=small,labelfont=bf]{caption}
\usepackage{subcaption}
\usepackage[hmargin=1.2in,vmargin=1.2in]{geometry}
\usepackage{authblk}
\usepackage{multirow}
\newtheorem{theorem}{Theorem}[section]

\newtheorem{corollary}[theorem] {Corollary}
\newtheorem{definition}[theorem]{Definition}

\newtheorem{lemma} [theorem]{Lemma}

\newtheorem{problem}[theorem]{Problem}
\newtheorem{proposition}[theorem]{Proposition}

\title{\bf Curling Numbers of Certain Graph Powers}
\author{C. Susanth}
\affil{\small Department of Mathematics\\ Research \& Development Centre\\ Bharathiar University\\ Coimbatore - 641046, Tamilnadu, India.\\ email:{\em susanth\_c@yahoo.com}} 
  
\author{Sunny Joseph Kalayathankal}
\affil{\small Department of Mathematics\\ Kuriakose Elias College\\ Kottayam - 686561, Kerala, India.\\ email:{\em sunnyjoseph2014@yahoo.com}}

\author{N. K. Sudev}
\affil{\small Department of Mathematics\\ Vidya Academy of Science \& Technology\\  Thrissur - 680501, Kerala, India.\\ email:{\em sudevnk@gmail.com}}

\author{K. P. Chithra}
\affil{\small Naduvath Mana, Nandikkara\\ Thrissur, India.\\ email:{\em chithrasudev@gmail.com}}

\author{Johan Kok}
\affil{\small Tshwane Metro Police Department\\ City of Tshwane, South Africa.\\ email:{\em kokkiek2@tshwane.gov.za}}

\date{}

\begin{document}
\maketitle
\newpage
\begin{abstract}
Given a finite nonempty sequence $S$ of integers, write it as $XY^k$, where $Y^k$ is a power of greatest exponent that is a suffix of $S$: this $k$ is the curling number of $S$. The concept of curling number of sequences has already been extended to the degree sequences of graphs to define the curling number of a graph. In this paper we study the curling number of graph powers, graph products and certain other graph operations.
\end{abstract}


\section{Introduction}
For terms and definitions in graph theory, we refer to \cite{BM1,CGT,ND,FH,DBW} and for more about different graph classes we refer to \cite{AVJ,JAG}.  All graphs we use here are simple, finite, connected and undirected.  The notion of curling number of integer sequences are introduced in \cite{BJJA} as follows. 

\begin{definition}{\rm
\cite{BJJA} Given a finite nonempty sequence $S$ of integers, write it as $XY^k$, where $Y^k$ is a power of greatest exponent that is a suffix of $S$: this $k$ is the curling number of $S$. }
\end{definition}
The concept of curling number of integer sequences has been extended to the degree sequences of graphs in \cite{KSC} and the corresponding properties and characteristics of certain standard graphs have been studied in that paper.  

Some of the main results on curling numbers are follows.
\begin{definition}{\rm
The \textit{Curling Number Conjecture} (see \cite{BJJA}) states that if one starts with any finite string, over any alphabet, and repeatedly extends it by appending the curling number of the current string, then eventually one must reach a 1.}
\end{definition}

\begin{definition}{\rm
\cite{KSC} A maximal degree subsequence with equal entries is called an \textit{identity subsequence}.  An identity subsequence can be a curling subsequence and the number of identity curling subsequences found in a simple connected graph $G$ is denoted $ic(G)$}
\end{definition}

\begin{theorem}
\cite{KSC} The number of curling subsequences of a simple connected graph $G$ is given by 
\[ \vartheta(G) = \left\{ \begin{array}{ll}
         1 & \mbox{; if $ic(G)=1$};\\
        $ic(G)+ic(G)!$ & \mbox{; otherwise}.\end{array} \right. \] 
\end{theorem}

\begin{theorem}
\cite{KSC} For the degree sequence of a non-trivial, connected graph $G$ on $n$ vertices, the curling number conjecture holds.
\end{theorem}

\begin{theorem}
\cite{KSC} If a graph $G$ is the union of $m$ simple connected graphs $G_i;1\leq i\leq m$ and the respective degree sequences are re-arranged as strings of identity subsequences, then
\[ cn(G) = \left\{ \begin{array}{ll}
         max\{cn(G_i)\} & \mbox{; if $X_i,X_j$ are not pairwise similar},\\
        max $$\sum_{i=1}^{m} k_i $$ & \mbox{; for all integer of similar identify subsequences}.\end{array} \right. \] 
\end{theorem}
As stated earlier, any degree sequence of a graph $G$ can be written as a string of identity curling subsequences.  In view of this fact, the concept of compound curling number of a graph $G$ has been introduced in \cite{KSC} as follows.

\begin{definition}{\rm
Let the degree sequence of the graph $G$ be written as a string of identity curling subsequences say, $X_1^{k_1} \circ X_2^{k_2} \circ X_3^{k_3} \ldots \circ X_l^{k_l}$.  The \textit{Compound curling number} of $G$, denoted $cn^c(G)$ is defined to be, $$cn^c(G)=\prod_{i=1}^{l} k_i.$$}
\end{definition}

The curling number and compound curling number of certain fundamental standard graphs have been determined in \cite{KSC} and the major results are listed in the following table.

\begin{table}[h!]
\begin{center}
\begin{tabular}{|c|l|c|c|}
\hline
\textbf{Sl.No.} & \textbf{Graph} & $cn$ & $cn^c$ \\
\hline
1 & Complete Graph $K_n, n\geq 1$ & $n$ & $n$ \\
\hline
2 & Complete Bipartite Graph $K_{m,n}, m\neq n$ & max $\{m,n\}$ & $mn$ \\
\hline
3 & Complete Bipartite Graph $K_{n,n}$ & $2n$ & $n^2$ \\
\hline
4 & Path Graph $P_n, n\geq 3$ & $n-2$  & $2(n-2)$ \\
\hline
5 & Cycle $C_n$ & $n$ & $n$ \\
\hline
6 & Wheel Graph $W_n=C_{n-1}+K_1$ & $n-1$ & $n-1$ \\
\hline
7 & Ladder Graph $L_n = P_n \times P_2, n\geq 2$ & $2(n-2)$ & $8(n-2)$ \\
\hline
\end{tabular}
\end{center}
\end{table}

\begin{proposition}
\cite{KSC} The compound curling number of any regular graph $G$ is equal to its curling number.
\end{proposition}

\section{New Results}

In this paper, we extend these studies on curling number to the integral powers of certain graph classes. By the size of a sequence, we mean the number of elements in that sequence.

\subsection{Curling number of certain graph powers}

Let us first recall the definition of integer powers of graphs as follows.

\begin{definition}{\rm
\cite{BM1} The \textit{ $r$-th power} of a simple graph $G$ is the graph $G^r$ whose vertex
set is $V$ , two distinct vertices being adjacent in $G^r$ if and only if their distance in $G$ is at most $r$. The graph $G^2$ is referred to as the square of $G$, the graph $G^3$ as the cube of $G$.}
\end{definition}

\noindent The following is an important theorem on graph powers.

\begin{theorem}
\cite{EWW} If $d$ is the diameter of a graph $G$, then $G^d$ is a complete graph.
\end{theorem}
If $d$ is the diameter of a graph $G$, then for any integer $q \geq d$, it can be noted that $G^q$ is a complete graph. Therefore we need to consider an integer $r\leq d$ to construct the $r$-th power of a given graph $G$.

From the table in the above section, it can be seen that the curling number of a complete graph is equal to the order of it.  Therefore for any graph $G$ of diameter $d$, we have $cn(G^r)=|V(G)|$, where $r\geq d$.  Hence, we consider an integer $r$, where $1\leq r \leq d$, as the power of a given graph $G$ in our present study.

\vspace{0.3cm}

In the following theorem we determine the curling number of the powers of path graphs.   

\begin{theorem}
For $n\geq 3$, let $P_n$ be a path on $n$ vertices and let $r\leq n-1$ be a positive integer, then the curling number of the $r$th power of $G$ is given by
\[ cn({P_n}^r) = \left\{ \begin{array}{ll}
         2 & \mbox{; if $r=\lfloor \frac{n}{2}\rfloor $};\\
        n-2r & \mbox{; if $r <\lfloor \frac{n}{2}\rfloor-1$};\\
        2(r+1)-n & \mbox{; if $\lfloor \frac{n}{2}\rfloor \le r \le n-1$}.\end{array} \right. \] 
\end{theorem}

\begin{proof}
Let $P_n$ be a path on $n$ vertices and $r$ be a positive integer less than or equal to $n$.  Here we have to consider the following two cases.

\noindent Case - 1: Assume that $r\leq \lfloor \frac{n}{2}\rfloor$.  The degree sequence of $P_n^r$ is $(r,r+1,r+2,\ldots,2r-1, \underbrace {2r,2r,\ldots,2r, 2r}_\text{(n-2r) terms},2r-1,\ldots r+2,r+1,r)$.
Therefore, the degree sequence becomes $(r,r+1,r+2,r+3,\ldots,2r-1)^2 \circ (2r)^{n-2r}$.
If $n-2r<2$, the curling number of $P_n^r$=2 and if $n-2r\geq2$, then the curling number of $P_n^r$ is $n-2r$. 

\noindent Case - 2: Assume that $r>\lfloor \frac{n}{2}\rfloor$.  The degree sequence of $P_n^r$ is $(r,r+1,r+2,\ldots,2r-3, \underbrace {2r-2,2r-2,\ldots, 2r-2}_\text{2(r+1)-n terms},2r-3,\ldots,r+2,r+1,r)$.
Therefore, the degree sequence becomes $(r,r+1,r+2,r+3,\ldots,2r-3)^2\circ (2r-2)^{2(r+1)-n}$.
That is, in this case, the curling number of $P_n^r$ is $2(r+1)-n$.  Therefore, we have
\[ cn({P_n}^r) = \left\{ \begin{array}{ll}
         2 & \mbox{; if $r=\lfloor \frac{n}{2}\rfloor $};\\
        n-2r & \mbox{; if $r <\lfloor \frac{n}{2}\rfloor-1$};\\
        2(r+1)-n & \mbox{; if $\lfloor \frac{n}{2}\rfloor \le r \le n-1$}.\end{array} \right. \]  
This completes the proof.
\end{proof}

\begin{corollary}
Let $P_n$ be a path on $n$ vertices and let $r\leq n$ be a positive integer, then the compound curling number of the $r$th power of $G$ is given by
\[ cn^c({P_n}^r) = \left\{ \begin{array}{ll}
         2^r & \mbox{; if $r=\lfloor \frac{n}{2}\rfloor $};\\
        2^r(n-2r) & \mbox{; if $r <\lfloor \frac{n}{2}\rfloor -1$};\\
        2^{r-1}(r+1)-n & \mbox{; if $\lfloor \frac{n}{2}\rfloor \le r \le n-1$}.\end{array} \right. \]
\end{corollary}
\begin{proof}
The proof of the result is immediate from the degree sequence, as explained in the above theorem.
\end{proof}

\begin{proposition}
The curling number is invariant under the $r$-th power of a cycle $C_n$, for all $r$ such that $1\leq r \leq \lfloor \frac{n}{2}\rfloor$. 
\end{proposition}
\begin{proof}
For $2\leq r\leq \lfloor \frac{n}{2}\rfloor$, we can see that the $r$th power of any cycle $C_n$ is $2r$-regular graph on $n$ vertices and hence, the graph $C_n^r$ has the degree sequence $(2r)^n$. Therefore, $cn(C_n^r)=n=cn(C_n)$, for any positive integer $1\leq r \leq \lfloor \frac{n}{2}\rfloor$.
%Hence, the curling number any power of $C_n$ is the same  as that of $C_n$.  }
%Is this condition applicable to every regular graph?  The following theorem answers this question.  In a similar way we can check for path graphs.}
\end{proof}

Another interesting graph class we consider here is the class of tadpole graphs, which are combinations of paths and cycles. The \textit{$(m,n)$ - tadpole graph}, also called a dragon graph, is the graph obtained by joining a cycle graph $C_m$ to a path graph $P_n$ with a bridge. The curling number of different powers of tadpole graphs are determined in the following theorem.

\begin{theorem}
For a tadpole graph $G=T_{m,n}$, $cn(G^r)=m+n-2(2r-1)$.
\end{theorem}
\begin{proof}
Let $U_0^{(r)}$ be the maximal identity sequence in the $r$th power of the tadpole graph $G$.  Clearly, the curling number of $G^r$ is the number of elements in the identity sequence $U_0^{(r)}$.  If $r$=1, for a tadpole graph $G=T_{m,n}$, we have the degree sequence is $(3)^1\circ (1)^1\circ (2)^{m+n-2}$. 

%When the value of $r$ increases by 1, we get consecutive powers of $G$.  
In each case, we need to analyse the size of the maximal identity sequence $U_0^{(r)}$.  When $r$ increases by $1$, it can be noted that the following changes in the degree sequence.  Two vertices of the cycle $C_m$ and one vertex of $P_n$, other than the common vertex $v$ come out from $U_0^{(r)}$.  Also, one end vertex of $P_n$, other than the end vertex comes out of $U_0^{(r)}$.  

Therefore, the total number of vertices that are excluded from the maximal identity sequence $U_0^{(r)}$ will be $(2(r-1)+1)+(2r-1) = 2(2r-1)$.  We can see that all other vertices of $G^r$ will have the same degree $2r$.  Therefore, the number of vertices in $G^r$ with degree $2r$ is $m+n-2(2r-1)$.  Therefore the curling number, $cn(G^r)$ of $G^r$ of a tadpole graph $T_{m,n}$ is $m+n-2(2r-1)$.
\end{proof}

\noindent The following result provides the compound curling number of a tadpole graph.

\begin{corollary}
For a tadpole graph $G=T_{m,n}$, the compound curling number, $cn^c(G^r)=r(r-1)(m+n-2(2r-1))$.
\end{corollary}
\begin{proof}
The degree sequence of $G^r$ is of the form $(3)^{r-1}\circ (1)^r \circ (2r)^{m+n-2(2r-1)}$.  Therefore the compound curling number,$cn^c(G^r)=r(r-1)(m+n-2(2r-1))$.
\end{proof}

Another fundamental graph structure that arouses much interest for our study in this context is a tree.  Since the adjacency and incidence patterns of various trees are uncertain, and hence a study on the curling number of arbitrary trees is very complex,  we proceed with trees having specific patterns.

At first, we study the curling number of complete binary trees in the following theorem.

\begin{theorem}
The curling number of any integral power of a complete binary tree of height $h\ge 2$ is $2^h$.  Also, the compound curling number of any integral power of a complete binary tree, $cn^c(G^r)$ is $2^{^{(h+1)}C_2}$.
%$\prod\limits_{i=0}^{h}2^i$. 
\end{theorem}
\begin{proof}
A complete binary tree of height $h$ has $2^{h+1}-1$ vertices such that the root vertex has degree 2, the internal vertices have degree 3 and external vertices have degree 1.  On taking higher powers, we notice that the degree of the vertices at the same level have the same degree.  Let $r_i, 0\leq i\leq h$ be the degree of vertices at the $i$th level.  Therefore the degree sequence of the complete binary tree $G$ can be written in the string form as $\prod\limits_{i=0}^{h}r_i^{2^i}$.  Therefore the curling number of $G^r$ is $2^h$.
It is also clear from the above expression that the compound curling number, $cn^c(G^r)=\prod\limits_{i=0}^{h}2^i = 2^{^{(h+1)}C_2}$, where $nC_r$ is the binomial coefficient. 
\end{proof}

As a generalisation of the above theorem, the curling number of a complete $n$-ary tree is determined as given in the following result. 
%of we can determine  be generalised for a complete $n$-ary tree of height $h$ as follows.
\begin{theorem}
The curling number of integral powers of a complete $n$-ary tree of height $h$ is $n^h$ and the compound curling number of integral powers of a complete $n$-ary tree of height $h$ is $\prod\limits_{i=0}^{h}n^i$. 
\end{theorem}
The proof of the theorem is immediate from the above theorem by taking $n$ in place of 2.

A Caterpillar is a particular type of tree, the study of whose curling numbers seems to be promising in this context. A \textit{caterpillar} is an acyclic graph which reduces to a path on removing its end vertices. The Caterpillar $G$ can be considered as the corona $P_n \odot K_1$.  A result on the curling number of an arbitrary caterpillar is described as follows.  
%Another special kind of tree that makes much interesting is for the study of curling number of caterpillar graphs The following therem .

\begin{lemma}
For a caterpillar graph $G$, $cn(G)=\max\{\eta,\sum_{i=0}^{n}l_i \}$ where $\eta$ is the maximum number of times a positive integer appears as a degree of internal vertices of $G$.
\end{lemma}
\begin{proof}
Let $V_1={u_i:1\leq i \leq n}$ is the set of internal vertices of $G$ and let $V_2$ be the set of all end vertices of $G$. Let $S_0$ be the degree sequence of the vertices in $V_2$ of $G$ and let $\eta$ be the maximum number of times a particular number that appears as a degree of vertices in $V_1$. For $1\leq i \leq n$, let $l_i$ be the number of end vertices adjacent to a vertex $u_i$ in $V_1(G)$.  Therefore the degree sequence of the vertices in $V_2$ can be written as $(1)^{\sum_{i=0}^{n}l_i}$.  Therefore clearly the curling number of $G$, $cn(G)= \max\{\eta,\sum_{i=0}^{n}l_i \}$.
\end{proof}
\begin{lemma}
For a caterpillar graph $G$, $cn(G)=\max\{\eta,\sum_{i=1}^{n} d(u_i)-2(n-3)\}$ where $\eta$ is the maximum number of times a positive integer appears as a degree of internal vertices of $G$.
\end{lemma}
\begin{proof}
The degree sequence of the vertices of $V_1$ is known and if $\eta$ is the maximum number of times a number appears as the degree of vertices in $V_1$.  Then the curling number of the caterpillar $G = \max\{\eta, d(u_i)-2(n-3)\}$.
For both vertices $u_1$ and $u_n$, the values $d(u_1)-1$ and $d(u_n)-1$ represent the number of leafs attached to $u_1$, $u_n$ respectively.  Similarly, for the vertex $u_i, 2\leq i \leq n-1$, the values $d(u_i)-2$ represents the number of leaves attached to the vertex $u_i$.  Therefore the maximum number of leaves = $\sum_{i=1}^{n}d(u_i)-2(n-3)$.
\end{proof}
Even though we can find out the compound curling number of specific caterpillar graph, calculation of the compound curling number of an arbitrary caterpillar graph $G$ is very complex because of insufficient knowledge about the degree of the internal vertices of $G$.

\section{Conclusion}{\rm
In this paper, we have discussed the curling numbers of certain graph classes and their powers.  There are several problems in this area which seem to be promising for further
investigations.  Some of the open problems we have identified during our study are the following.

\begin{problem}{\rm
Determine the curling number and the compound curling number of an arbitrary binary tree on $n$ vertices.}
\end{problem}
\begin{problem}{\rm
Determine the curling number and the compound curling number of an $n$-ary tree on $n$ vertices.}
\end{problem}
\begin{problem}{\rm
Determine the curling number and the compound curling number of an arbitrary tree on $n$ vertices.}
\end{problem}
\begin{problem}{\rm
Determine the curling number and the compound curling number of an arbitrary caterpillar.}
\end{problem}
The concepts of curling number and compound curling number of certain graph powers and discussed certain properties of these new parameters for certain standard graphs. More problems regarding the curling number and compound curling number of certain other graph classes, graph operations, graph products and graph powers are still to be settled. All these facts highlights a wide scope for further studies in this area.

\end{document}